\documentclass[reqno]{amsart}
\usepackage{amssymb}
\usepackage{amsmath}
\usepackage{amsfonts}
\usepackage{amsthm}
\usepackage{mathrsfs}
\usepackage[english]{babel}
\usepackage[table]{xcolor}
\usepackage{enumitem}
\usepackage[final, colorlinks=true, citecolor=blue]{hyperref}
\usepackage{url}
\usepackage{tabularx}

\def\smat#1{\left(\begin{smallmatrix} #1 \end{smallmatrix} \right)}
\def\pmat#1{\left(\begin{matrix} #1 \end{matrix} \right)}

\def\ll#1{{\left\langle{#1}\right\rangle}}

\newcommand{\eps}{\varepsilon}
 \newcommand{\lifteps}{\varepsilon}
\newcommand{\mat}{\left(\begin{smallmatrix} \alpha & \beta\\ \gamma& \delta
\end{smallmatrix}\right)}
\newcommand{\gns}{\Gamma_{ns}}
\newcommand{\gnsp}{\Gamma^{+}_{ns}}
\newcommand{\cD}{\mathcal{D}}

\newcommand{\OO}{\mathcal{O}}
\newcommand{\Z}{\mathbb{Z}}
\newcommand{\Q}{\mathbb{Q}}
\newcommand{\set}[2]{\left\{{#1} \, : \, {#2}\right\}}

\newcommand{\GL}{\operatorname{GL}}
\newcommand{\HH}{\mathcal{H}}
\newcommand{\qns}{\mathcal{Q}_{ns}}

\newcommand{\Qns}{\mathcal{Q}_{ns,D}}
\newcommand{\qD}{\mathcal{Q}_{D}}

\newcommand{\qnsps}{\mathcal{Q}_{ns,D,s}}

\newcommand{\pnsps}{\mathcal{P}_{ns,D,s}}

\newcommand{\cocl}{L_{ns}^\vee / L_{ns}}
\newcommand{\Lns}{\mathscr{L}_{ns}}

\setenumerate{label={$($\normalfont\arabic*)}}
\DeclareMathOperator{\SL}{SL}

\DeclareMathOperator{\Pic}{Pic}

\DeclareMathOperator{\M}{M}
\DeclareMathOperator{\Tr}{tr}

\DeclareMathOperator{\Gal}{Gal}

\DeclareMathOperator{\tr}{tr}
\DeclareMathOperator{\nm}{nm}
\DeclareMathOperator{\adj}{adj}

\DeclareMathOperator{\HCl}{HeegCl}
\def \QQ{\mathbb{Q}}

\def \ZZ{\mathbb{Z}}

\def \RR{\mathbb{R}}
\def \CC{\mathbb{C}}
\def\<#1>{{\left\langle{#1}\right\rangle}}

\numberwithin{equation}{section}

\theoremstyle{plain}
\newtheorem{prop}[equation]{Proposition}
\newtheorem{thm}[equation]{Theorem}

\newtheorem{lemma}[equation]{Lemma}
\newtheorem*{thm*}{Theorem}
\theoremstyle{remark}

\newtheorem{remark}[equation]{Remark}

\begin{document}

\title
[A GKZ theorem for non-split Cartan curves]
{A Gross-Kohnen-Zagier theorem \\ for non-split Cartan curves}

\author{Daniel Kohen}

\address{Universit\"at Duisburg-Essen, Fakult\"at f\"ur Mathematik - Germany}
\email{daniel.kohen@uni-due.de}
\thanks{DK was supported by an Alexander von Humboldt postdoctoral fellowship}

\author{Nicol\'as Sirolli}

\address{Universidad de Buenos Aires and IMAS-CONICET - Argentina}
\email{nsirolli@dm.uba.ar}
\thanks{}


\begin{abstract}
Let $p$ be a prime number and let $E$ be rational an elliptic curve of conductor
$p^2$ and odd analytic rank. 
We prove that the positions of its special points arising from non-split Cartan
curves and imaginary quadratic fields where $p$ is inert are encoded in the
Fourier coefficients of a Jacobi form of weight $6$ and lattice index of rank
$9$, obtaining a result analogous to that of Gross, Kohnen and Zagier.

\end{abstract}				

\maketitle

\section{Introduction}

Let $E/\Q$ be an elliptic curve of conductor $N$ and odd analytic rank. Let $D$
be a negative fundamental discriminant prime to $2N$. Assume that $p$ splits in
$\Q(\sqrt{D})/\Q$ for every $p \mid N$ (the Heegner condition).
For each $r \in \ZZ/2N$ such that $r^2 \equiv D \pmod{4N}$, there exists a
special cycle on the Jacobian of the modular curve $X_0^*(N)$, which under the
modular parametrization maps to a point $Q_{D,r}^* \in E(\Q) \otimes \Q$.

In their celebrated article \cite{GKZ}, Gross, Kohnen and Zagier prove that the
points $Q_{D,r}^*$ are aligned, and moreover, that their positions in the line
are given by the Fourier coefficients of a Jacobi form $\phi_E$  of weight $2$ and index
$N$. Moreover, they show that the space generated by the points $Q_{D,r}^* $ is
non-trivial if and only if $L'(E/\Q,1) \neq 0$.

Later on, Zhang showed in \cite{zhang_gz} how to obtain special
points on general Shimura curves under more relaxed Heegner
conditions, thus giving flexibility in the choice of the discriminants.

Consider the case $N = p^2$ where $p$ is an odd prime. Let $D$ be as above but
now assume that $p$ is \emph{inert} in $\Q(\sqrt{D})/\Q$. Then the construction
of Zhang gives points $Q_{D,s}^+ \in E(\Q) \otimes \Q$ by considering special cycles
on the Jacobian of the non-split Cartan curve $X_{ns}^+(p)$.  
In this context, the main result of this article is the following analogue of
\cite[Theorem C]{GKZ}.

\begin{thm*}[Theorem \ref{thm:main_elliptic_curve}]
	Let $E/\QQ$ be an elliptic curve of conductor $p^2$ and odd analytic rank.
	There exists a positive definite even lattice $(\Lns,\beta)$ of rank $9$ and
	a Jacobi form $\psi_E$ of weight $6$ and lattice index $\Lns$ such that for
	every negative discriminant $D$ which is a non-square modulo $p$ we have
	\[
		Q_{D,s}^+= c_{\psi_E}\left(\beta(s)-D/4p^2,s\right) \, Q,
	\]
	for some  $Q \in E(\Q) \otimes \Q$  which is non-zero if and only if
	$L'(E/\QQ,1) \neq 0$.
\end{thm*}

Borcherds gave in \cite{borcherds} an astonishing generalization of the result from
Gross, Kohnen and Zagier. He considers Heegner divisors
associated to certain even lattices, and proves that their generating series is modular. 
He then obtains (the ``ideal statement'' of) \cite[Theorem C]{GKZ} by applying
his modularity result to a specific lattice. 

In order to use the result of Borcherds we need to find a concrete lattice whose
Heegner divisors are the special cycles coming from $X_{ns}^+(p)$, at least for
discriminants prime with $p$. 
This lattice is obtained by carefully studying the action of the non-split
Cartan group on the set of quadratic forms giving rise to the special points. 
It turns out to not be stably isomorphic to a rank $1$ lattice, which explains
why our main result involves Jacobi forms of higher rank lattice index. 

In the classical case the Jacobi form $\phi_E$ is, by definition, a Hecke
eigenform with the same eigenvalues as the modular form attached to $E$.
We are able to prove the analogous statement for $\psi_E$ combining Theorem
\ref{thm:main_elliptic_curve} with the formal distribution relations satisfied
by the Heegner points on the non-split Cartan curves.

In addition, as an application of \cite{Kohen2}, we relate the Fourier
coefficients of $\psi_E$ with certain coefficients of the Jacobi form $\phi_E$.
This gives a link with the classical theory, which is a particularity of the
Atkin-Lehner sign of $E$ at $p$ being equal to $1$.

We also focus on the explicit computation of the Jacobi form $\psi_E$, as
this was one of the main motivations for this article.

\medskip

This article is organized as follows. 
Firstly, we review the non-split Cartan curves. 
In Section \ref{sect:specialp} we consider their special points and their
relation with certain quadratic forms, which we study carefully in the next
section.
In Section \ref{sect:specialc} we define the relevant special cycles and
consider the action of Hecke and Atkin-Lehner operators on them. 
Afterwards, we recall the results from \cite{zhang_gz} which show that the
non-split Cartan curve uniformizes the elliptic curve, and that give a formula
for the height of the special cycles in this context. 
In Section \ref{sect:vvjf} we review some facts about vector valued and Jacobi
modular forms, and the connection between them.
In the following section we explain how to build up the desired lattice in order
to apply the results from Borcherds, relying on the results of Section
\ref{sect:quad_forms}.
In Section \ref{sect:main_results} we prove the main results of this article,
including Theorem \ref{thm:main_elliptic_curve}. 
In the penultimate section we give the connection with classical Jacobi forms.
We end this article with an explicit example, checking the validity of the
conclusions we achieved.

\section{Setting}

Let $p$ be an odd prime, and let $\lifteps \in \ZZ/4p^2$ be a non-square
modulo $p$ such that $\lifteps \equiv 1 \pmod{4}$.

Let  $E/\QQ$ be an elliptic curve of conductor $p^2$ and odd analytic rank. In
particular, its local Atkin-Lehner sign $w_p$ at $p$ equals $1$. 

Let  $K$ be an imaginary quadratic field such that $p$ is \emph{inert} in $K$
and let $\OO \subseteq \OO_K$ be an order of discriminant $D$ inside the ring of
integers of $K$. 
We will always assume (unless explicitly stated) that $D$ is not divisible by $p$.

\section{Non-split Cartan curves}
The non-split Cartan order is
\[
	M_{ns} = \set{\pmat{a & b \\ c & d} \in M_{2}(\ZZ)}
	{a \equiv d,\, b\varepsilon \equiv  c \pmod{p}}.
\]
The non-split Cartan group modulo $p$ is
\[
	C_{ns} = \set{\overline{M} \in  \GL_{2}(\mathbb{F}_p)}{M \in M_{ns}},
\]
where $\overline{M}$ denotes the reduction modulo $p$ of $M$. 
This is an abelian subgroup of $ \GL_{2}(\mathbb{F}_p)$ isomorphic to
$\mathbb{F}^{\times}_{p^2}$.
We also have the corresponding arithmetic non-split Cartan group $\gns=M_{ns}
\cap \GL^{+}_{2}(\QQ)$.

We also consider
\[
	M^+_{ns} = M_{ns} \cup
	\set{\pmat{a & b \\ c & d} \in \\M_{2}(\ZZ)}
	{a \equiv -d,\, b\varepsilon \equiv  -c \pmod{p}},
\]
and, as before,
\[
	C^+_{ns} = \set{\overline{M} \in  \GL_{2}(\mathbb{F}_p)}{M \in M^+_{ns}}.
\]
This group is the normalizer of $C_{ns}$  inside $\GL_{2}(\mathbb{F}_p) $, and
the quotient $C^+_{ns}/C_{ns}$ is of order $2$. 
This normalizer is a maximal subgroup of $\GL_{2}(\mathbb{F}_p) $.
In addition, we set  $\gnsp=M^+_{ns} \cap \GL^+_{2}(\QQ)$. Thus, the
elements of $\gnsp$ normalize $\gns$ and the quotient
$\gnsp/\gns$ is also of order $2$.

We consider the non-split Cartan (open) modular curve $Y_{ns}=\gns \backslash
\mathcal{H}$ (resp. $Y^+_{ns}=\gnsp \backslash \mathcal{H}$). 
Its compactification is given by $X_{ns}=\gns \backslash \mathcal{H^*}$ (resp.
$X^+_{ns}=\gnsp \backslash \mathcal{H^*}$), where $\mathcal{H^*}$ is the union
of the upper half plane $\mathcal{H}$ and the cusps.  
Furthermore, we let
$J_{ns}$ (resp. $ J^+_{ns}$) denote the Jacobian of $X_{ns}$ (resp. $X^+_{ns}$).

\medskip

The following result allows us to work with cycles of not necessarily degree
$0$ if we are willing to kill the torsion.

\begin{prop}\label{prop:pics_iso}
	The natural map $J_{ns} \otimes \Q \to \Pic(X_{ns}) \otimes \Q$
	is an isomorphism.
\end{prop}

\begin{proof}
	Since the natural map $J_{ns} \to \Pic(X_{ns})$ is injective, we only need
	to prove that after tensoring with $\Q$ it is surjective.
	Let $\xi \in \Pic(X_{ns}) \otimes \Q$  be the Hodge class \cite[Section
	6.2]{zhang_gz}, which is a certain divisor of degree
	$1$ that is a rational linear combination of cusps.
	By the Manin-Drinfeld theorem there exists $M \geq 1$ such that $M \xi = 0$.
	Given $\cD \in \Pic(X_{ns})$, let 
	$\cD' = \cD - \deg(\cD) \xi \in J_{ns} \otimes \Q$. 
	Then $\cD'\otimes M$ maps to
	$\cD \otimes M - \deg(\cD) \xi \otimes M = \cD \otimes M$.
\end{proof}

\section{Special points on non-split Cartan curves}\label{sect:specialp}

An embedding $\iota:K \hookrightarrow M_{2}(\QQ)$ is an \emph{optimal embedding} of
$\OO$ into $M_{ns}$ if
\[\iota(K) \cap M_{ns} = \iota(\OO). \]
Any such embedding gives rise to a unique point $z_{\iota} \in \mathcal{H}$
fixed under the action of $\iota(K^{\times})$. 
The point $z_{\iota} \in X_{ns}(\CC)$ will be called a \emph{special point} (in
some situations also called CM or Heegner point) on the non-split Cartan curve.
It is known that $z_\iota \in X_{ns}(H_{\OO})$, where $H_{\OO}$ denotes the
ring class field associated to the order $\OO$ (see for example \cite[Prop
3.1]{Kohen2}). 

Although this definition of special points has the
merit of being very general, it is often convenient to have a concrete
description of these points in terms of
quadratic forms, as in \cite{MR803364}.

Fix $\omega \in K$ such that $\OO = \Z + \omega \Z$.
Let $\iota$ be an embedding of $\OO$ into $M_{ns}$. Write
\begin{equation}\label{eqn:iotaomega}
	\iota(\omega) = \pmat{a&b\\c&d}.
\end{equation}
Since $\iota(\omega) \in M_{ns}$ we know that $a,b,c,d \in \ZZ$ and
that $a-d \equiv c-b \eps \pmod{p}$.
In addition, we have that $D=\tr(\omega)^2 - 4 \nm(\omega)=  (a-d)^2 + 4bc$.
If we set 
\[
	A=c,\,  B= a-d,\, C=-b,
\]   
and we let $v = [A,B,C]$ be the quadratic form given by $AX^2+BXY+CY^2$, we know
that $B \equiv A + C \eps \equiv 0 \pmod{p}$ and that $v$ has discriminant $D$. 

We then consider the sets of integral quadratic forms
\begin{align*}
	\qD & = \set{[A,B,C]}{ B^2-4AC = D}, \\
	\qns & = \set{[A,B,C]}{ B \equiv A+C\eps \equiv 0 \pmod p}, \\
	\Qns & = \qns \cap \qD.
\end{align*}

\begin{prop} \label{prop:identification}

	The map $\iota \mapsto v$ gives a bijection between the set of
	embeddings of $\OO$ into $M_{ns}$ and $\Qns$. Furthermore, 
	$\iota$ is optimal if and only if $v$ is primitive.

\end{prop}

\begin{proof}

	To compute the reverse map, given $v = [A,B,C] \in \Qns$ we let
	$\iota(\omega)$ be the matrix given by
	\begin{equation}\label{eqn:iota_wf}
		\iota(\omega)= 
		\pmat{\frac{\tr(\omega)-B}{2} & -C\\ A& \frac{\tr(\omega)+B}{2}}.
	\end{equation}
	Since both $\tr(\omega)^2$ and $B^2$ are equal to $D$ modulo $2$, this
	matrix is integral, and it clearly  belongs to $M_{ns}$. Furthermore, it has
	the same trace and discriminant as $\omega$, hence this defines an
	embedding $\iota$ of $\OO$ into $M_{ns}$.

	Let $t=\gcd(A,B,C)$.
	If the embedding fails to be optimal, there exist
	$\lambda_1, \lambda_2 \in \QQ$ with $\lambda_2 \notin \ZZ$
	such that $\lambda_1+ \lambda_2 \iota(\omega) \in M_{ns}$. 
	Using the notation from \eqref{eqn:iotaomega}, we see that this implies that
	$\lambda_2(a-d),\lambda_2 b, \lambda_2 c \in \ZZ$.
	Then any prime dividing the denominator of $\lambda_2$ divides $t$,
	hence $v$ is not primitive.
	
	Conversely, if $t \neq 1$ let $\lambda_{1}=\frac{-a}{t}, \lambda_2=\frac1t$.
	Using that $t$ is not divisible by $p$ (otherwise $p$ would divide $D$), an
	immediate computation gives that $\lambda_1+ \lambda_2 \iota(\omega) \in
	M_{ns} \setminus \iota(\OO)$, showing that the embedding is not optimal.
\end{proof}

Given a quadratic form $v = [A,B,C]$ with negative discriminant, denote by $z_v$
the unique root in $\HH$ of the polynomial $A X^2 + B X + C$.

\begin{prop}\label{prop:raiz}

	Under the bijection above, $z_\iota = z_v$.

\end{prop}

\begin{proof}
	Since $\iota(\omega)$ fixes $z_\iota$, using \eqref{eqn:iota_wf} it is
	easy to see that $z_\iota$ is a root of $AX^2 + BX + C$.
\end{proof}

\section{Quadratic forms}\label{sect:quad_forms}

Consider the action of $\SL_2(\ZZ)$ on the set of integral quadratic forms,
which is given $[A,B,C] \cdot \mat = [A',B',C']$, where 
\begin{align}\label{eqn:acc_fc}
	A'&=A \alpha^2+ B\alpha \gamma + C \gamma^2, \nonumber \\
	B'&=A(2\alpha\beta)+ B(\alpha\delta+ \beta\gamma)+ C(2\gamma\delta), \\
	C'&=A\beta^2+ B\beta\delta+ C\delta^2. \nonumber
\end{align}

\begin{lemma} \label{lem:inv_ns}

Let $[A,B,C] \in \qns$ and let $ [A',B',C']= [A,B,C] \cdot M$ where $M \in
\Gamma_{ns}$.
Then, the following holds.

\begin{enumerate}

	\item $A' \equiv A \pmod{p}$.
	\item $B' \equiv B \pmod{2p}$.
	\item $C' \equiv C \pmod{p}$.
	\item $A-A' \equiv \lifteps (C-C') \pmod{p^2}$.

\end{enumerate}
In particular, $[A',B',C'] \in \qns$.

\end{lemma}

\begin{proof}
	Write $M = \mat$. 
	In order to prove $(1)$, using \eqref{eqn:acc_fc} we compute
    \[
		A' = A \alpha^2 + B \alpha \gamma+ C\gamma^2 \equiv A \alpha^2 -
	A/\varepsilon \gamma^2 \equiv A \pmod{p}.
	\] 
In the first congruence we have used that $B \equiv A + C \varepsilon \equiv 0
\pmod{p}$ and in the second one we have used the fact that $\mat \in
\Gamma_{ns}$.
Items $(2)$ and $(3)$ follow similarly. 

In order to prove $(4)$, first note that the condition is equivalent to
$A'-C'\varepsilon \equiv A-C\varepsilon \pmod{p^2}$. We compute, again using
\eqref{eqn:acc_fc},
\[
	A' - C' \eps =
	A(\alpha^2 - \eps\beta^2) + B(\alpha\gamma - \eps\beta\delta)
	+ C(\gamma^2 - \eps\delta^2).
\]
Since $B\equiv \alpha\gamma - \eps \beta \delta \equiv 0 \pmod p$ we write
\[
	A' - C' \eps \equiv A-C\varepsilon +  A(\alpha^2 - \eps\beta^2-1) + 
	C(\gamma^2 - \eps\delta^2+\varepsilon) \pmod{p^2}.
\]
As $(\alpha^2 - \eps\beta^2-1)$ is divisible by $p$ we can replace $A$ with $-C
\varepsilon$ without changing the value of the expression modulo $p^2$. Hence we
get
\[
	A' - C' \eps \equiv A-C\varepsilon +
	C( \gamma^2 - \eps\delta^2+\varepsilon - \eps \alpha^2+{\eps}^2\beta^2+\eps)
	\pmod{p^2}.
\]
Completing squares we have
\[
	\gamma^2 - \eps\delta^2+\varepsilon - 
	\eps \alpha^2+{\eps}^2\beta^2+\eps
	\equiv (\eps\beta-\gamma)^2-\eps(\alpha-\delta)^2- 2\alpha\delta\eps+
	2\beta\gamma\eps +2\eps \pmod{p^2}
\]
Because $\alpha\delta-\beta\gamma=1$ and $\alpha \equiv \delta, \beta\eps \equiv
\gamma \pmod{p}$ we obtain that this last expression is $0$ modulo $p^2$ and
thus
\[
	A' - C' \eps \equiv A-C\varepsilon  \pmod{p^2},
\]
as we wanted.
\end{proof}

In consequence it is natural to consider, for $s \in \ZZ/2p^2$,  the set
\[
	\qnsps =
	\set{[A,B,C] \in \Qns}{A-C\lifteps \equiv s \pmod{p^2}, \, 
	B \equiv s \pmod 2}.
\]    
By Lemma \ref{lem:inv_ns} the group $\Gamma_{ns}$ acts on $\qnsps$.

\begin{prop}\label{prop:biy_qs}
The set $\qnsps$ is non-empty if and only if 
$s^2 \equiv \lifteps D \pmod{4p^2}$.
In that case the inclusion from $\qnsps$ to $\mathcal{Q}_{D}$ gives a
bijection
\[
	\qnsps/ \gns \stackrel{\simeq}{\longrightarrow} \mathcal{Q}_{D}/ \SL_{2}(\ZZ).
\]

\end{prop}

\begin{proof}
Let $[A,B,C] \in \qnsps$. We have that
\[
	\lifteps D  = \lifteps(B^2-4AC)  \equiv
	\lifteps B^2+ (A-C\lifteps)^2- (A+C\lifteps)^2 \equiv s^2 \pmod{4p^2},
\]
which proves the ``only if'' part of the first claim.

We now prove that the map is injective. Let $[A,B,C],[A',B',C'] \in
\mathcal{Q}_{ns,D,s}$ and suppose there exists a matrix $M =\mat \in \SL_2(\ZZ)$
such that $[A,B,C] \cdot M =[A',B',C']$.
We know that $A+C\varepsilon \equiv A'+C'\varepsilon \equiv 0 \pmod{p}$ and
$A-C\varepsilon \equiv A'-C'\varepsilon \equiv s \pmod{p}$, hence $A \equiv
A', C \equiv C' \pmod{p}$. Since $B$ and $B'$ are divisible by $p$ and
$A,A',C,C'$ are not (because $D$ is not divisible by $p$), by
\eqref{eqn:acc_fc} we have 
\begin{align*}
	A' \equiv A \pmod p & \implies \alpha^2- \gamma^2/ \varepsilon \equiv 1
	\pmod{p}, \\
	C' \equiv C \pmod p & \implies \delta^2- \beta^2 \varepsilon \equiv 1
	\pmod{p}.
\end{align*}
Combining these congruences and using that $\det M = 1$ we get
\[
	(\alpha-\delta)^2- \varepsilon(\beta-\gamma/\varepsilon)^2 \equiv  \alpha^2-
	\gamma^2/ \varepsilon + \delta^2- \beta^2 \varepsilon +
	2(\beta\gamma-\alpha\delta) \equiv 2-2 \equiv 0 \pmod{p}.
\]
Since $\varepsilon$  is a non-square modulo $p$, we have $\alpha \equiv \delta,
\gamma \equiv \beta \varepsilon \pmod{p}$, and thus $M$ belongs to
$\Gamma_{ns}$ as desired.

Now, assuming that $s^2 \equiv \lifteps D \pmod{4p^2}$, we prove the
surjectivity, and in particular that $\qnsps$ is non-empty.
Let $[A,B,C] \in \mathcal{Q}_{D}$. The matrices 
\[
	\pmat{B/2& -C\\ A& -B/2},
	\pmat{0& s/2\varepsilon\\ s/2& 0}
	\quad \in\GL_2(\mathbb{F}_p)
\]
are conjugate, since they have the same characteristic polynomial, which has
simple roots.
Moreover, as $\smat{0& s/2\varepsilon\\ s/2& 0} \in C_{ns}$, which is
abelian, and $\det: C_{ns} \rightarrow \mathbb{F}^{\times}_p$ is surjective,
these matrices are actually conjugate by a matrix $\overline{M} \in
\SL_2(\mathbb{F}_p) $.  
Take a lift $M \in \SL_2(\Z)$, and let $[A',B',C']=[A,B,C] \cdot M$.
By construction,  $B'\equiv s \pmod{2}$ and $A'-C' \lifteps  \equiv s
\pmod{p}$.
Now, both $A'-C' \lifteps$ and $s$ are square roots of $D \lifteps$ modulo $p^2$
which are congruent (and non-zero) modulo $p$, and thus they must be equivalent
modulo $p^2$. 
This implies that $[A',B',C'] \in \qnsps$.
\end{proof}

Recall that any element of $\gnsp \setminus \gns$ defines an involution
$W_p$ on $X_{ns}$.  In the same fashion as the $\Gamma_{0}(N)$ case, this
involution interchanges the square roots of $ \lifteps D \pmod{p^2}$. 

\begin{lemma} \label{lem:inv_ns+}

Let $[A,B,C] \in \qns$ and let $ [A',B',C'] = [A,B,C] \cdot M$, where $M \in
\gnsp \setminus \gns$.
Then, the following holds.

\begin{enumerate}

	\item $A' \equiv - A \pmod{p}$.
	\item $B' \equiv -B \pmod{2p}$.
	\item $C' \equiv -C \pmod{p}$.
	\item $A-A' \equiv -\lifteps (C-C') \pmod{p^2}$.

\end{enumerate}
In particular, $\qnsps \cdot W_p = \mathcal{Q}_{ns,D,-s}$.

\end{lemma}
	
\begin{proof}
The proof is exactly the same as in Lemma \ref{lem:inv_ns}, but as we are taking
$M=\mat \in \gnsp \setminus \gns$ we have the congruences $\alpha \equiv -\delta
\pmod{p} , \gamma \equiv -\eps\beta \pmod{p}$. Finally, note that acting by $M$ we change
the value of $s$ to $-s$ modulo $p^2$.
\end{proof}

\section{Special cycles}\label{sect:specialc}

For the convenience of the reader, we recall the definitions
\begin{align*}
	\qD & = \set{[A,B,C]}{ B^2-4AC = D}, \\
	\qns & = \set{[A,B,C]}{ B \equiv A+C\eps \equiv 0 \pmod p}, \\
	\Qns & = \qns \cap \qD, \\
	\qnsps & =
	\set{[A,B,C] \in \Qns}{A-C\lifteps \equiv s \pmod{p^2}, \, 
	B \equiv s \pmod 2}.
\end{align*}
We consider the sets of positive definite quadratic forms
\begin{align*}
	\mathcal{P}_D& =  \set{[A,B,C] \in  \mathcal{Q}_{D}}{A>0}, \\
	\pnsps & =  \set{[A,B,C] \in  \qnsps}{A>0}, 
\end{align*}
and the special cycle
\[
	P_{ns,D,s} = \sum_{v \in \pnsps / \Gamma_{ns}} z_v \qquad
	\in \Pic(X_{ns}),
\]
which we denote by $P_{D,s}$ for aesthetical purposes.
It is understood to be zero if the set $\pnsps$ is empty.

\begin{remark}

	Since $p$ is inert in $K$, there exists $s \in \ZZ/2p^2$ such that $s^2
	\equiv \lifteps D \pmod{4p^2}$. In particular, by Proposition
	\ref{prop:identification}, the set $\pnsps$ is non-empty.
	Therefore under our setting the special cycles $P_{D,s}$ are meaningful.

\end{remark}

The formula for the action of the Hecke operators $T_{\ell}$ on these special
cycles is the same one that appears in \cite[p.507]{GKZ} for the classical
modular curve.

\begin{prop}\label{prop:hecke_ciclos}
	Let $\ell \neq p$ be a prime number. Then 
	\begin{equation}\label{eqn:hecke_pd}
		T_\ell P_{D,s} =
		P_{D\ell^2,s\ell} + \genfrac(){}{0}{D}{\ell}P_{D,s}
		+ \ell \,P_{D/\ell^2, s/\ell}\, ,
	\end{equation}
	where the last term is understood to be zero if $\ell^2$ does not divide
	$D$.
\end{prop}

\begin{proof}
	This follows from \cite[Corollary 6.6]{Cornut-Vatsal}. The computation in
	\cite{Cornut-Vatsal} is of local nature and works for general Shimura curves
	of level not divisible by $\ell$, hence it applies in our setting.
\end{proof}

In addition, we have the following formulas for the action of $W_p$ and complex
conjugation (denoted by a bar) on these special cycles.

\begin{prop}\leavevmode \label{prop:conjugar}  
We have that
\[
	P_{D,s} \cdot W_p = P_{D,-s}=\overline{ P_{D,s}}.
\]
\end{prop}

\begin{proof}
Since the action of $W_p$ is given by a matrix of $\SL_2(\ZZ)$, which preserves
positiveness, Lemma \ref{lem:inv_ns+} implies that $\pnsps \cdot W_p =
\mathcal{P}_{ns,D,-s}$, proving the first equality. 
The second equality is \cite[Proposition 3.3 (iii)]{MR3484373}.
\end{proof}

The following proposition tell us that for \emph{fundamental} discriminants the
cycles $P_{D,s}$ are obtained as the traces of the special points on the
non-split Cartan curve. 
In order to state it, we denote by $H$ the Hilbert class field corresponding
to the maximal order $\OO_K$ and, assuming that $\pnsps$ is non-empty, we let
$\iota$ be an optimal embedding corresponding via the bijection of Proposition
\ref{prop:identification} to an element in $\pnsps$.
Moreover, we recall that under the identification of Proposition
\ref{prop:pics_iso} we can view $P_{D,s}$ as an element of $J_{ns} \otimes \QQ$.

\begin{prop}\label{prop:pds_racional}
	Suppose that $D$ is fundamental. Then
	\[
		P_{D,s} = \sum_{\sigma \in \Gal(H/K)} \sigma \cdot z_\iota.
	\]
	In particular, $P_{D,s} \in J_{ns}(K) \otimes \Q$.
\end{prop}

\begin{proof}
By Proposition \ref{prop:biy_qs} we have a bijection $\pnsps/ \gns \simeq
\mathcal{P}_D / \SL_2(\Z)$.
Since $D$ is fundamental, the forms in $\mathcal{P}_D $ are necessarily
primitive  and then the set $\mathcal{P}_D
/ \SL_2(\Z)$ is in bijection with $\Gal(H/K)$. 
Unraveling these identifications and using Proposition \ref{prop:raiz}, the
result follows. 
\end{proof}

\section{Modular parametrization and Zhang's formula}

Recall that $E/\QQ$ is an elliptic curve of conductor $p^2$ and such that $w_p=1$.
The following proposition shows that $E$ is uniformized by the non-split Cartan
curve. 

\begin{prop}
There exists a non-zero Hecke-equivariant rational map $\pi^+: J^+_{ns} \to E$.
\end{prop}

\begin{proof}
	The existence of a such a map $\pi : J_{ns} \to E$ follows from
	\cite[Theorems 1.2.2, 1.3.1]{zhang_gz}.  Since $w_p=1$ this map factors
	through $J^+_{ns}$.
\end{proof}

We define $P_{D,s}^+$ as the image of $P_{D,s}$ on $J^+_{ns}(K) \otimes
\QQ$, and the corresponding point $Q_{D,s}^+=\pi^+(P_{D,s}^+)$ on $E(K) \otimes
\QQ$.  

\begin{remark} \label{rmk:smenoss}

	The  special cycles $P_{D,s}^+$ (and hence the points $Q_{D,s}^+$) do not
	depend on $s$, since $P_{D,-s}^+ = P_{D,s}^+$ and whenever $\lifteps D$ is
	a non-zero square in $\Z/2p^2$ it has exactly two square roots.
	Nevertheless, as this is a particularity of working with curves of level
	$p^2$, we consider the dependence on $s$ throughout the article.

\end{remark}

\begin{prop}\label{prop:hecke_qd}

Let $\ell \neq p$ be a prime number. Then,
\begin{equation}\label{eqn:hecke_qd}
	a_\ell(E) \, Q_{D,s}^+ = 
	Q_{D\ell^2,s\ell}^+ + \genfrac(){}{0}{D}{\ell}Q_{D,s}^+
	+ \ell \, Q_{D/\ell^2, s/\ell}^+,
\end{equation}
where the last term is understood to be zero if $\ell^2$ does not divide $D$.

\end{prop}

\begin{proof}
	This follows by applying $\pi^+$ to \eqref{eqn:hecke_pd}  and using the
	Hecke equivariance of this map.
\end{proof}

\begin{prop}\label{prop:racional}

	The points $Q_{D,s}^+ $ belong to $E(\Q) \otimes \Q$.

\end{prop}

\begin{proof}
If $D$ is fundamental, by Proposition \ref{prop:conjugar} we have
\[
	2P^+_{D,s}=P_{D,s}+ P_{D,s}\cdot W_p= P_{D,s} + \overline{ P_{D,s}}.
\]
Therefore, using Proposition \ref{prop:pds_racional} we obtain that $P_{D,s}^+
\in  J_{ns}(\Q) \otimes \Q$, and since $\pi^+$ is
rational, we get that $Q_{D,s}^+ \in E(\Q) \otimes \Q$. 
Using \eqref{eqn:hecke_qd} we see that the result holds for non-fundamental
discriminants as well.
\end{proof}

In this setting we have the following generalization of the Gross-Zagier formula
by Zhang.

\begin{thm} \cite[Theorem 1.2.1]{zhang_gz} \label{thm:zhang}
	Assume that $D$ is fundamental and that  $s^2 \equiv \lifteps D \pmod{4p^2}$. 
	Then
	\[
		L'(E/K,1)= c \, \sqrt{D} \, \ll{Q_{D,s}^+, Q_{D,s}^+},
	\]
	where $\ll{\cdot, \cdot}$ denotes the N\'eron-Tate height pairing and $c$ is a
	non-zero constant which does not depend on $D$.
\end{thm}

\section{Vector valued and Jacobi modular forms}\label{sect:vvjf}

Let $(L,\beta)$ be an even lattice and let $L^{\vee}$ denote its dual. 
The symmetric bilinear form $\beta$ induces an
integral quadratic form given by $\beta(x)=\frac12 \beta(x,x)$. 
Its reduction modulo $1$ induces a quadratic form on the finite group
$L^{\vee}/L$ with values in $\QQ/\ZZ$, called the \emph{discriminant form}.
Let $V_L$ be the group ring $\CC[L^{\vee}/L]$ with standard basis $\left\{
e_s \right\}_{s \in L^{\vee}/L}$. 
Let $\rho_L$ be the representation of the metaplectic cover of $\SL_2(\ZZ)$
associated to the discriminant form $(L^{\vee}/L,\beta)$ (see \cite[Section
2]{borcherds}), and let $\rho_L^*$ denote its dual.

Given $k \in \frac12 \ZZ$, we say that a holomorphic function
$f: \mathcal{H} \rightarrow V_L$ is a \emph{vector valued modular form} of
weight $k$ and type $\rho_L^*$ if it is invariant under the $k$-slash operator
induced by $\rho_L^*$ (see \cite[Definition 3.49]{MR3309829})
and is holomorphic at $\infty$, i.e. if we consider the Fourier expansion
\[
	f(\tau)= \sum_{s \in L^{\vee}/L} \sum_{m \in \QQ} c_{m,s} \, q^m e_s,
	\qquad q = e^{2 \pi i \tau},
\]
then $c_{m,s} \neq 0$ implies that $m \geq 0$.
The space of such forms will be denoted by $M_{k}(\rho_L^*)$, and is determined
by the discriminant form and $k$.
The space of cuspforms, i.e. those forms for which $c_{m,s} \neq 0$ implies
that $m>0$, will be denoted by $S_{k}(\rho_L^*)$.

\medskip

Assume furthermore that $(L,\beta)$ is \emph{positive definite}.
Given $k \in \Z$, we say that a holomorphic function $\psi: \mathcal{H} \times
(L\otimes \CC) \rightarrow \CC$ is a \emph{Jacobi form} of weight $k$ and
lattice index $L$ if it is invariant under the $k$-slash operator induced by the
action of the Jacobi group of $(L,\beta)$ (see \cite[Definition
3.29]{MR3309829}),
and is holomorphic at $\infty$, i.e. if we consider the Fourier expansion
\[
	\psi(\tau,z)= \sum_{s \in L^{\vee}} \sum_{n \in \ZZ}
	c(n,s) \, q^n  e^{\beta(s,z)}
\]
then $c(n,s) \neq 0$ implies that $n \ge \beta(s)$.
The space of Jacobi forms of weight $k$ and lattice index $L$ will be denoted by
$J_{k,L}$.
The space of Jacobi cuspforms, i.e. those forms for which $c(n,s) \neq 0$ implies
that $n > \beta(s)$, will be denoted by $S_{k,L}$.

By \cite[Theorem 3.5]{MR3309829} we can identify Jacobi forms with vector valued
modular forms.
More precisely, if we denote by $r$ the rank of the lattice $L$, the following holds.

\begin{prop}\label{prop:jac_vv}

Given $\psi \in J_{k,L}$ and $m \in \QQ_{\ge 0}$ let 
$c_{m,s}=c(\beta(s)+m,s)$ if $m+\beta(s)\in \ZZ$ and $c_{m,s}=0$ otherwise.
Then $c_{m,s}$ depends only on the class of $s$ modulo $L$.  
Furthermore, the map
\begin{equation*}
	\psi(\tau,z)= \sum_{s \in L^{\vee}} \sum_{n \in \ZZ}
	c(n,s) \, q^n  e^{\beta(s,z)}  \mapsto
	\sum_{s \in L^{\vee}/L} \sum_{m \in \QQ} c_{m,s} \, q^m e_s
\end{equation*}
gives an isomorphism $J_{k+r/2,L} \simeq M_{k}(\rho^*_{L})$, which preserves
cuspforms.

\end{prop}

In our applications we will need to work with vector valued modular forms
associated to a non-positive definite lattice, which in principle do not
correspond to Jacobi forms.
In order to solve this, we resort to the concept of stably isomorphic lattices.
Two (even) lattices $L_1, L_2$ are said to be stably isomorphic if there exist
(even) unimodular lattices $U_1,U_2$ such that
$L_{1} \oplus U_1 \simeq  L_{2} \oplus U_2$.
In particular, the discriminant forms of $L_1$ and $L_2$ are the same and
therefore
\begin{equation}\label{eqn:vv12}
	M_{k}\left(\rho_{L_1}^*\right) \simeq M_{k}\left(\rho_{L_2}^*\right).
\end{equation}
Now \cite[Theorem 1.10.1]{nikulin1979integral} tell us that for any given
lattice we can find a positive definite one that is stably isomorphic to it.
Combining this with the aforementioned results, any vector valued modular form
can be thought as a Jacobi form.

\section{Finding the correct lattice}

We let $V = \set{x \in M_2(\Q)}{\Tr x = 0}$, and we let $\beta$ be the bilinear
form in $V$ given by $\beta(x,y) = - \Tr(x \adj(y))/4p^2$.
We consider the lattice
\[
	L = \set{\pmat{B& 2C\\ -2A& -B}}{A,B,C \in \Z}.
\]
We let $\SL_2(\Z)$ act on $V$ by conjugation. Then the map
\[
	[A,B,C] \mapsto \begin{pmatrix} B& 2C\\ -2A& -B \end{pmatrix}
\]
gives a $\SL_2(\Z)$-equivariant bijection between the set of integral binary
quadratic forms and $L$.
Under this bijection the set $\qns$ corresponds to the lattice
\[
	L \cap M_{ns} = 
	\set{\pmat{B& 2C\\ -2A& -B} \in L}{B \equiv 0 \pmod{p}, A \equiv -
	C\varepsilon \pmod{p}}.
\]

The invariants of the action of $\Gamma_{ns}$ on $\qns$ given by Lemma \ref{lem:inv_ns}
suggest considering the lattice $L_{ns}$ given by

\begin{multline}\label{eqn:lns}
	L_{ns} = \bigg\{ \pmat{B& 2C\\ -2A& -B} \in L: A
	\equiv B \equiv C \equiv 0 \pmod{p}, \\
	B \equiv 0 \pmod{2}, A \equiv \lifteps C \pmod{p^2} \bigg\}.
\end{multline}
Notice that $L_{ns}$ is an even lattice of signature $(2,1)$. It has the
following  properties.

\begin{prop} \leavevmode \label{prop:ldualmodl}

\begin{enumerate}

	\item The dual lattice $L_{ns}^{\vee}$ is equal to $L \cap M_{ns}$.
	\item The lattices $L_{ns}$ and $L_{ns}^\vee$ are $\gnsp$-invariant, and the
		action of $\gns$ on $\cocl$ is trivial.
	\item  The discriminant form $(\cocl, \beta)$ is isomorphic to
		$\left(\ZZ/2p^2, s \mapsto \frac{s^2}{4\lifteps p^2}\right)$.

\end{enumerate}

\end{prop}

\begin{proof} \leavevmode

\begin{enumerate}

\item
	Take $x_1= \smat{B_1& 2C_1\\ -2A_1& -B_1}, \,
	x_2= \smat{B_2& 2C_2\\ -2A_2& -B_2} \in  V$.
	Then
\[
	\beta(x_1,x_2) =  \frac{2 B_1B_2 -4A_1C_2-4C_1A_2}{4p^2}.
\]
For this to be integral for every $x_1 \in L_{ns}$ we need that $B_{2} \equiv 0
\pmod{p}$ and $A_{1}C_{2}+ C_{1}A_{2} \equiv 0 \pmod{p^2}$.  We know that
$A_{1}=p \alpha$ and $C_{1}=p \gamma$ for some integers $\alpha,\gamma$.
Moreover, since $A_1 \equiv C_{1} \varepsilon \pmod{p^2}$ we get that $\alpha
\equiv \gamma \varepsilon \pmod{p}$ and
that
\[ 0 \equiv \alpha C_{2} + \gamma A_{2} \equiv  \gamma( C_2 \varepsilon + A_{2}) \pmod{p}.  \]
Since we can choose some $\gamma$ that is non-zero modulo $p$ we obtain that
$C_{2} \varepsilon = -A_2 \pmod{p}$. Thus we have proved that
\[
	L_{ns}^{\vee} = \set{\pmat{B& 2C\\ -2A& -B}}
	{A,B,C \in \ZZ,\, B  \equiv 0 \pmod{p},\, A \equiv -C\varepsilon \pmod{p}},
\]
which is equal to $L \cap M_{ns}$.

\item This follows immediately by Lemmas \ref{lem:inv_ns} and \ref{lem:inv_ns+}.

\item

	The map $L_{ns}^{\vee} \to \ZZ/ 2 \times \ZZ/p^2$ given by
	\[
		\pmat{B& 2C\\ -2A& -B} \mapsto 
		\left( B \pmod{2}, A-C \lifteps \pmod{p^2}\right)
	\]
	has kernel $L_{ns}$, and is surjective.
	In fact, given $s \in \ZZ/2p^2$, take $\smat{B& 2C\\ -2A& -B} \in
	L_{ns}^{\vee}$ such that $B \equiv s \pmod 2$ and $ A-C \lifteps \equiv s
	\pmod{p^2}$. 
	As $\lifteps$ was chosen to be $1$ modulo $4$ and $B \equiv A+C\lifteps
	\equiv 0 \pmod{p}$ we obtain
	\[
		B^2-4AC \equiv \left(B^2+ (A-C\lifteps)^2-(A+C\lifteps)^2\right)/\lifteps 
		\equiv s^2/\lifteps \pmod {4p^2}.
	\]
	This finishes the proof because
	$\beta\left(\smat{B& 2C\\ -2A& -B}\right)=\frac{B^2-4AC}{4p^2}$. 
\end{enumerate}
\end{proof}

\section{Main results}\label{sect:main_results}

Under the identification given by Proposition~\ref{prop:pics_iso}, for $d \in
\QQ_{< 0}$ and $s \in \cocl$ we consider the cycle 
\[
	Y_{d,s}= \sum_{\set{v \in s+L_{ns}}{\beta(v)=d}/\gns} 
	z_v \qquad \in J_{ns} \otimes \Q.
\]
Under the identification given by part (3) of Proposition~\ref{prop:ldualmodl},
we see that if $Y_{d,s} \neq 0$ then $s^2 \equiv D \lifteps
\pmod{2p^2}$, where $D /4p^2= d$ (this is valid for all negative discriminants,
not necessarily  prime to $p$). 
If $d \ge 0$ we set $Y_{d,s}=0$.

We have the following modularity result, which is a consequence of the work
Borcherds \cite{borcherds}.

\begin{prop} \label{prop:modularity}
	The generating series 
	\[
		\sum_{s \in L_{ns}^{\vee}/L_{ns}}
		\sum_{m \in \QQ}  
		Y_{-m,s} \, q^m e_{s}
	\]
	is modular. More precisely, it lies in the space
	$(J_{ns} \otimes \Q) \otimes S_{3/2}\left(\rho^*_{L_{ns}}\right)$.
\end{prop}

\begin{proof}
	We follow \cite[Section 4]{borcherds} closely, using its notation.

	Let $G(L_{ns})$ be the Grassmanian of positive definite planes in
	$L_{ns} \otimes \RR$, which is identified with $\mathcal{H}$. 
	Under this identification, given $v \in L_{ns}^\vee$ of negative norm we
	have that $z_v$ is the image in $\gns \backslash \HH$ of the element of
	$G(L_{ns})$ orthogonal to $v$. For $d<0$ we consider the Heegner divisors
	$y_{d,s}$ given by
	\[
		y_{d,s}= \sum_{\set{v \in s+L_{ns}}{\beta(v)=d}/\gns} 
		z_v \qquad \in \Pic\left(X_{ns}\right).
	\]
	Furthermore, we define a formal symbol $y_{0,0}$, and if we write $y_{d,s}$
	with either $d>0$ or $d=0$ and $s \neq 0$ we understand that this is zero. 

	Let $\Gamma(L_{ns})$ denote the group of automorphisms of the lattice $L_{ns}$
	that act trivially on $\cocl$.
	The Heegner class group $\HCl(X_{ns})$ is the group generated by the
	divisors $y_{d,s}$ quotiented by the subgroup of the so called principal
	Heegner divisors, which are the divisors of the form $c_{0,0}
	\,y_{0,0}+\cD$; here $c_{0,0}$ is an integer and $\cD$ is the divisor of a
	meromorphic automorphic form of weight $c_{0,0}/2$ with respect to $\Gamma(L_{ns})$
	and some finite order character (the finite order condition is explained in
	the correction \cite{MR1788047}).  

	Given this setting, \cite[Theorem 4.5]{borcherds}, combined with the fact
	that $S_{3/2}\left(\rho^*_{L_{ns}}\right)$ has a basis of forms with
	rational coefficients (see \cite[Theorem 5.6]{MR1981614}), tell us that the
	generating series
	\[
		\sum_{s \in L_{ns}^{\vee}/L_{ns}}
		\sum_{m \in \QQ}  
		y_{-m,s} \, q^m e_{s}
	\]
	belongs to
	$\left(\HCl(X_{ns})\otimes \QQ\right) \otimes
	S_{3/2}\left(\rho^*_{L_{ns}}\right)$.
	In order to conclude, it suffices to prove that the natural map
	\[
		\HCl(X_{ns})\otimes \QQ \longrightarrow J_{ns} \otimes \Q
	\]
	given by $y_{-m,s}  \mapsto Y_{-m,s}$ is well defined.

	As the cuspform $\Delta$ of weight $12$  has no zeros and no poles on
	$\mathcal{H}$, the element $24 \, y_{0,0}$ is a principal Heegner divisor
	and thus the element $y_{0,0} \in \HCl(X_{ns})\otimes \QQ$ is trivial.
	Then it is enough to show that the divisor of a weight $0$ meromorphic
	automorphic form for $\Gamma(L_{ns})$ is mapped to $0 \in J_{ns} \otimes \Q$ (we
	can omit the finite character condition since we have tensored with $\Q$).

	Part (2) of Proposition \ref{prop:ldualmodl} claims precisely that $\gns
	\subseteq \Gamma(L_{ns})$, so a weight $0$ meromorphic automorphic form for
	$\Gamma(L_{ns})$ is in particular automorphic for $\gns$, and therefore its
	divisor is zero on $J_{ns} \otimes \Q$, as we wanted to prove.
\end{proof}

In order to close the circle we need to show that the Fourier coefficients can
be described in terms of special points. 
From now on we return to the original setting where $D$ is prime to $p$. 

\begin{prop} \label{prop:special cycles}
	
	Denote by $Y^+_{d,s}$ the projection of $Y_{d,s}$ onto $J^+_{ns}\otimes \QQ$.
	Let $d = D/4p^2$ and $s \in \cocl$.
	Then
	\[
		Y^+_{d,s} = 2 P_{D,s}^+ \quad \in J^+_{ns}\otimes \QQ.
	\]

\end{prop}

\begin{proof}
	Unraveling the definitions, we see that there is a bijection
	\begin{align*}
		\mathcal{Q}_{ns,D,s}  & \longrightarrow
		\set{v \in s + L_{ns}}{\beta(v) = d} \\
		[A,B,C] & \mapsto \pmat{B & 2C \\ -2A & -B}.
	\end{align*}
	If $[A,B,C]$ is an indefinite form in $\qnsps$ then $[-A,-B,-C] \in
	\mathcal{P}_{ns,D,-s}$, and they give rise to the same point in $\mathcal{H}$.
	Therefore, by Proposition \ref{prop:conjugar},
	\[
		Y_{d,s} =P_{D,s}+P_{D,-s}=P_{D,s}+ P_{D,s}\cdot W_p
		\quad \in J_{ns}\otimes \QQ.
	\]
	Then the result follows by projecting onto $J^+_{ns}\otimes \QQ$.
\end{proof}

As remarked by Borcherds \cite[Example 5.1]{borcherds}, in the $\Gamma_0(N)$
case the corresponding lattice splits as the direct sum of a lattice generated
by a vector of norm $2N$ and an even unimodular hyperbolic $2$ dimensional
even lattice. 
Therefore the space of vector valued modular forms with that lattice index
corresponds to the space of Jacobi form of weight $2$ and index $N$.  

However, for $L_{ns}$ this is not the case anymore.  In fact, using the criteria
by Nikulin for the existence of even lattices with given signature and
discriminant form (see \cite[Theorem 1.10.1]{nikulin1979integral}), part (3) of
Proposition \ref{prop:ldualmodl}, and the fact that $\eps$ is not a square
modulo $p$ we see that $L_{ns}$ is \emph{not} stably isomorphic to an even
positive definite lattice of rank $1$.

However, \cite[Corollary 1.10.2]{nikulin1979integral} implies that $L_{ns}$ is
stably isomorphic to an even positive definite lattice $\Lns$ of rank $9$
(and this is the smallest possible dimension as the signature of the discriminant
form is invariant modulo $8$).

\medskip

Piecing together the results of this section we obtain the following theorem.
\begin{thm}  \label{thm:main}

	There exists $\psi \in (J^+_{ns}\otimes \Q) \otimes S_{6,\Lns}$ such
	that for every negative discriminant $D$ prime to $p$ we have
	\[
		P_{D,s}^+ = c_\psi(\beta(s)-D/4p^2,s).
	\]
\end{thm}

\begin{proof}
	By Proposition \ref{prop:jac_vv} and \eqref{eqn:vv12}
	we have that
	\[
		S_{6,\Lns} \simeq  S_{3/2}\left(\rho^*_{\Lns}\right)
		\simeq S_{3/2}\left(\rho^*_{L_{ns}}\right).
	\]
	Then the result is a straightforward combination of Propositions
	\ref{prop:modularity} and \ref{prop:special cycles}.
\end{proof}

The next step is to project onto the elliptic curve $E$ in order to recover a
statement about the special points on it.
Given a special cycle $P_{D,s}^+$, recall that $Q_{D,s}^+=\pi^+(P_{D,s}^+)  \in
E(\Q) \otimes \Q$ (Proposition \ref{prop:racional}). 
The following is the  main result of this article, which is analogous to
\cite[Theorem C]{GKZ}.

\begin{thm} \label{thm:main_elliptic_curve}
	
	There exists a Jacobi form $\psi_E \in S_{6,\Lns}$
	such that for every negative discriminant $D$  prime
	to $p$ we have
	\begin{equation}\label{eqn:main_elliptic_curve}
		Q_{D,s}^+= c_{\psi_E}\left(\beta(s)-D/4p^2,s\right) \, Q,
	\end{equation}
	for some  $Q \in E(\Q) \otimes \Q$  which is non-zero if and only if
	$L'(E/\QQ,1) \neq 0$.

\end{thm}

\begin{proof}

	If $D$ is fundamental by Theorem \ref{thm:zhang} we have that $Q_{D,s}^+
	\neq 0$ if and only if $L'(E/K,1) \neq 0$.
	In particular, if $L'(E/\QQ,1)=0$ all these cycles vanish. 
	Furthermore, by Proposition \ref{prop:hecke_ciclos} this is also true if $D$
	is not fundamental.
	Then \eqref{eqn:main_elliptic_curve} holds letting $Q = 0$ and $\psi_E = 0$.
	
	On the other hand, if $L'(E/\QQ,1) \neq 0$, then $E(\QQ) \otimes \QQ$ has rank 1. 
	Letting $Q$ be a generator, we identify $E(\QQ) \otimes \QQ$ with $\QQ$ and we define
	$\psi_E=\pi^+ \psi  \in  S_{6,\Lns}$. The result follows immediately applying 
	Theorem \ref{thm:main}.
\end{proof}

Finally, we prove that the form $\psi_E$ lies in the expected Hecke eigenspace.

\begin{prop} \label{prop:hecke_iso}
	The Jacobi form $\psi_E$ is an eigenform for $T_{\ell}$ for $\ell \neq p$
	with eigenvalues $a_{\ell}(E)$.
\end{prop}

\begin{proof}

	Let $\chi = T_{\ell}\psi_E-a_{\ell}(E)\psi_E \in S_{6,\Lns}$. 
	Combining Proposition \ref{prop:hecke_qd} with Theorem
	\ref{thm:main_elliptic_curve} and the explicit formulas for the action of
	the Hecke operators on Jacobi forms (\cite[Theorem 2.6.1]{ajouz2015hecke}),
	we easily check that the Fourier coefficients of $\chi$ associated to
	negative discriminants $D$ not divisible by $p$ vanish.

	Using Proposition \ref{prop:jac_vv} we consider $\chi$ as a form in
	$S_{3/2}\left(\rho_{L_{ns}}^*\right)$.
	Let $H$ be the unique subgroup of $\cocl$ of order $p$ which
	corresponds to the multiples of $2p$ inside $\ZZ/2p^2$.
	The group $H$ is isotropic, meaning that $\beta$ vanishes on $H$. In
	addition, $H^{\bot}$ corresponds to the subgroup of multiples of $p$ (which
	in turn correspond to the discriminants divisible by $p$).
	Then $\chi$ is supported in $H^{\bot}$, meaning that the components
	$\chi_s$ for $s \not \in H^{\bot}$ vanish. 

	Let $M=H^{\bot}/H$. Then by \cite[Proposition 3.3]{MR3119226} the form
	$\chi$ is an oldform, arising from the space of vector valued modular forms
	of weight $3/2$ and discriminant form $M$. Since $M$ has order $2$ this space
	corresponds to the space of classical Jacobi form of weight $2$ and index $1$. 
	Because the latter space is trivial we have that $\chi=0$, as we wanted to prove. 
\end{proof}

\begin{remark}

	In \cite{GKZ} the authors obtain the Jacobi form $\psi_E$ using the
	correspondence between systems of eigenvalues of Jacobi forms and classical
	modular forms proved in \cite{MR958592}. 
	In particular, they do not need to prove (the analogous result of)
	Proposition~\ref{prop:hecke_iso}; moreover, they use this correspondence to
	extend (the analogous result of) Theorem~\ref{thm:main_elliptic_curve} to
	non-fundamental discriminants, knowing that it holds for fundamental ones.
	We want to stress that we are doing the exact opposite: we already know that
	Theorem~\ref{thm:main_elliptic_curve} holds for all (not divisible by $p$)
	discriminants, and we leverage this to obtain that $\psi_E$ is a Hecke
	eigenform.
\end{remark}

\section{Relation with classical Jacobi forms}

Denote by $\phi_E$ the classical Jacobi form of weight $2$ and scalar index $p
^2$ which corresponds to (the modular form corresponding to) $E$ via the
Skoruppa-Zagier lift. 
The following proposition, valid since $w_p=1$, relates the coefficients of
$\psi_E$ with certain coefficients of $\phi_E$.

\begin{prop} \label{prop:clasicas}
There exists a non-zero constant $\kappa$ such that for every negative
fundamental $D$  and $s \in \ZZ/2p^2$ with  $s^2 \equiv \lifteps D \pmod{4p^2}$,
\begin{equation}\label{eqn:classical_puntos}
	Q_{Dp^2,r}^*=\kappa \, Q_{D,s}^+,
\end{equation}
where $r \equiv p^2s \pmod{2p^2}$.
Furthermore,
\begin{equation}\label{eqn:classical}
	c_{\phi_E}\left((r^2-Dp^2)/4p^2,r\right) 
	= \kappa \, c_{\psi_E}\left(\beta(s)-D/4p^2,s\right).
\end{equation}
\end{prop}

\begin{proof}
The first part is a special case of \cite[Theorem 4.3]{Kohen2}. 
More precisely, using the notation from \cite{Kohen2}, if we take $M=f=1$, then the
special point $\tr^{H_p}_K(P_p)$ (resp. $ \tr^H_K(P_1)$) is, by construction, our
$Q^*_{Dp^2,r}$, with $r \equiv p^2s \pmod{2p^2}$ (resp. our  $Q_{D,s}^+$). 
Furthermore, the constant $\kappa$ such that $\tilde P_1 = \kappa \, P_1$
does not depend on $D$.
Then \eqref{eqn:classical_puntos} follows, since
$\tr^{H_p}_H(P_p) = \tilde P_1$.

As $p$ is inert in $K$, a quadratic form corresponding to $Q_{Dp^2,r}^*$ must be
necessarily primitive.
Hence \eqref{eqn:classical} follows from the interpretation of \cite[Theorem
C]{GKZ} by Borcherds (see \cite[Example 5.1]{borcherds}) and Theorem
\ref{thm:main_elliptic_curve}.
\end{proof}

\begin{remark}

The results of this article can be generalized in a straightforward manner to
include elliptic curves of conductor $N=p^2M$ with $p \nmid M$ and odd
analytic rank, giving the positions of the Heegner points induced by quadratic
imaginary fields in which $p$ is inert but every prime dividing $M$ is split.
For such curves we can certainly have $w_p=-1$.
Although Theorem \ref{thm:main_elliptic_curve} continues to hold, the
positions of the points in the line will not necessarily be given by the
coefficients of a classical Jacobi form, as in Proposition \ref{prop:clasicas}.

\end{remark}

\section{An explicit example}\label{sect:example}

We let $p = 17$ and $\lifteps = 5$.
Consider the elliptic curve
\[
	E: \quad y^2+xy+y=x^3-x^2-199x+510.
\]
It has conductor $289=17^2$ and rank $1$. 
Our goal is to compute the Jacobi form $\psi_E$ alluded to in
Theorem~\ref{thm:main_elliptic_curve}.

\medskip

A generator of $E(\Q)$, up to torsion, is given by the point $Q=[-12,38]$.
We compute for negative discriminants $D$ prime to $17$ such that
$17$ is inert in $\Q(\sqrt{D})$ the special points $Q_{D,s}^+$ (which, according
to Remark \ref{rmk:smenoss}, do not depend on $s$) and we give the integer
$m(D)$ such that, up to torsion, $Q_{D,s}^+=m(D) \, Q$.

These points can be computed using the non-split Cartan curve as explained
in \cite{Kohen}, \cite{MR3484373}. They are constructed by giving an explicit
modular parametrization $\pi^+:X^+_{ns} (17) \rightarrow E$, which amounts to 	
finding an explicit cuspform for $\Gamma^+_{ns}(17)$ with the same eigenvalues 
as $f_E$ for the Hecke operators, where $f_E$ is the cuspform of weight $2$ and
level $289$ corresponding to $E$.  
In Table \ref{tab:mds} below we record the computations for all valid
discriminants of absolute value less than $200$.

\medskip

We now compute $\psi_E$, using \texttt{SAGE} (\cite{sagemath}).
We first compute the positive definite lattice $\Lns$ which
is stably isomorphic to the lattice $L_{ns}$ given by \eqref{eqn:lns}. 
This is done by using \cite[Algorithm 2.3]{raum2016computing}, which 
adds to $L_{ns}$ a copy of $E_8$ (the positive definite unimodular lattice of
rank $8$) and splits a two dimensional hyperbolic lattice $U$. 
Concretely, we have
$L_{ns} \oplus E_8 \simeq \Lns \oplus U$.
As claimed above, $\Lns$ has rank $9$. Its Gram matrix is given by
\[
	\pmat{
		34 & -136 & -80 & 16 & -4 & -4 & 0 & 0 & 0 \\
		-136 & 578 & 323 & -68 & 17 & 17 & 0 & 0 & 0 \\
		-80 & 323 & 190 & -40 & 10 & 10 & 0 & 0 & 0 \\
		16 & -68 & -40 & 12 & -3 & -3 & 0 & 0 & 0 \\
		-4 & 17 & 10 & -3 & 2 & 0 & 0 & 0 & 0 \\
		-4 & 17 & 10 & -3 & 0 & 2 & -1 & 0 & 0 \\
		0 & 0 & 0 & 0 & 0 & -1 & 2 & -1 & 0 \\
		0 & 0 & 0 & 0 & 0 & 0 & -1 & 2 & -1 \\
		0 & 0 & 0 & 0 & 0 & 0 & 0 & -1 & 2
	}.
\]

Now we need to compute the space $S_{3/2}\left(\rho_{\Lns}^*\right)$.
We first, using (an enhanced version by Ehlen of) the algorithms of
Williams derived from \cite{williams2018poincare}, compute the space
$S_{12+3/2}\left(\rho_{\Lns}^*\right)$, since these algorithms do not work in
weight $3/2$. 
By \cite{bruinier2002rank}, this space has dimension $296$.
Then we look for forms in that space such that the first Fourier coefficient
vanishes for every $s \in \Lns^\vee / \Lns$ so that, when we divide by $\Delta$,
the modular discriminant, we obtain the space
$S_{3/2}\left(\rho_{\Lns}^*\right)$. This space, according to
\cite{ehlen2017computing}, has dimension $7$.
Finally, using the identification given by Proposition \ref{prop:jac_vv}, we
search in that space for forms which have their first Fourier coefficients given
by a fixed multiple of the values $m(D)$ computed in Table \ref{tab:mds}, and we
see that this space is indeed one dimensional, spanned by the Jacobi form
$\psi_E$.

We conclude our computations by verifying Proposition \ref{prop:clasicas}.
Using the algorithms derived from \cite{rsst} we compute the Fourier
coefficients of the Jacobi form $\phi_E$, which are given by applying the
dualized Skoruppa-Zagier lift to the modular symbol corresponding to $E$. We
verified that \eqref{eqn:classical} holds for every $D$ of absolute value less
than $200$.

\begin{table}[h]
	\caption{Relative positions of special points}
	\label{tab:mds}
    \begin{tabular}{| l | l | l | l | l | l |}
    \hline
   	$D$ &  $m(D)$ & $D$ &  $m(D)$ & $D$ &  $m(D)$ \\
	\hline
   	-3 & 1 & -71 & -3   &-139 & 1 	 \\
	-7&  -1  & -75 & 3  &-143 & -2  \\
	-11 & 3  & -79 & -1  &-147 & -5 \\
	-12 & 0 &-80 & -2  &-148 & 2  \\
	-20 & 2 &-88 & 2 &-156 & -8 \\
	-23 & 1  &-91 & -2  &-159 & 2  \\
	-24 & 2 & -92 & -2  &-160 & -2  \\
	-27& 0 & -95 & -4  &-163 & 7  \\
	-28 & 2 &-96 & -2  &-164 & -4 \\
	-31 & 1  &-99 & -3  &-167 & 1 \\ 
	-39 & 4 &-107 & -3  &-175 & -3  \\ 
	-40 & 2  &-108 & 0  &-176 & -6 \\ 
	-44 & 0 &-112 & 0  &-180 & -2  \\ 
	-48 & -2 &-116 & -2  &-184 & 8  \\ 
	-56 & 0 &-124 & -2 &-192 & 2  \\ 
	-63 & -1  &-131 & 3 &-199 & -1 \\ 
    \hline
    \end{tabular}
\end{table}

\section*{Acknowledgments} We would like to thank both Stephan Ehlen and Brandon
Williams for kindly helping us to use their algorithms for computing the Jacobi
form of Section \ref{sect:example}. 
We would also like to thank the anonymous referee for their useful comments.

\end{document}